\documentclass[12pt]{iopart}

\usepackage{a4wide,bm,nicefrac,enumerate,graphicx,times,psfrag,epsfig,color,subfigure,cite}
\usepackage{pstool}
\usepackage{amsfonts}
  \expandafter\let\csname equation*\endcsname\relax
  \expandafter\let\csname endequation*\endcsname\relax
\usepackage{amsmath,amsthm}
\usepackage{epstopdf}

\newtheorem{theorem}{Theorem}[section]
\newtheorem{lemma}[theorem]{Lemma}
\newtheorem{proposition}[theorem]{Proposition}
\newtheorem{algorithm}[theorem]{Algorithm} 
\newtheorem{corollary}[theorem]{Corollary} 
\theoremstyle{definition}

\theoremstyle{remark}
\newtheorem{remark}[theorem]{Remark}

\newcommand{\s}{\sigma(\bm A)}
\newcommand{\E}{\mathfrak{E}}
\newcommand{\diag}[1]{\mathrm{diag}\left(#1\right)}
\renewcommand{\L}{\mathcal{L}}
\newcommand{\dx}{\,\mathsf{d}\widehat x}
\newcommand{\x}{\widehat x}
\renewcommand{\O}{\Omega}
\newcommand{\sign}{\mathrm{sign}}

\begin{document}
\title{Computing the Entropy of a Large Matrix}

\date{Version~\today}

\author[T.~P.~Wihler]{Thomas P.~Wihler}
\address{Mathematics Institute, University of Bern, Sidlerstrasse 5, CH-3012 Bern, Switzerland}
\ead{wihler@math.unibe.ch}

\author[B.~Bessire]{B\"anz Bessire}
\author[A.~Stefanov]{Andr\'e Stefanov}
\address{Institute of Applied Physics, University of Bern, Sidlerstrasse 5, CH-3012 Bern, Switzerland}
\ead{andre.stefanov@iap.unibe.ch}
\ead{bbessire@iap.unibe.ch}

\begin{abstract}
Given a large real symmetric, positive semidefinite matrix, the goal of this paper is to show how a numerical approximation of the von Neumann entropy of the matrix can be computed in an efficient way, without relying on matrix diagonalization. An application from quantum optics dealing with the entanglement between photons illustrates the new algorithm.

\end{abstract}
\vspace{2pc}
\maketitle

\section{Introduction}

Entropy is a quantity of large relevance in different areas of physics. Being at the foundation of thermodynamics and information theory \cite{Shannon1948}, the concept of entropy is even capable to describe the information content of black holes \cite{Bekenstein1973}. Despite its well-known properties \cite{Wehrl1978}, a numerical computation of the entropy can be rather difficult. This is especially the case when the classical definition of entropy is applied in the quantum domain where the state of a system is generically described by a density operator. The von Neumann entropy \cite{neumann1927} of a quantum state allows to quantify the amount of disorder in a system and has shown to be a valid quantifier of entanglement between two subsystems \cite{BarnettSM1989,BarnettSM1991,vedral1997}. For pure states, the entropy of entanglement is defined as the von Neumann entropy of the reduced density operator \cite{Greenberger2009}, whereas for mixed states several definitions exist, e.g.~the relative entropy of entanglement \cite{Vedral2002}. Here, we are interested in computing the entropy of entanglement of pure states. As a concrete example, we consider in Section~\ref{sc:qo} an entangled two-photon state generated by spontaneous parametric down-conversion. Under certain conditions to be specified, the reduced density operator of such a state assumes the shape of a large matrix with a number of rows and columns of order $\mathcal{O}(10^{8})$. In this case, the common approach to calculate the entropy of entanglement by means of the reduced density operator's eigenvalue spectrum fails. The here presented numerical algorithm computes the entropy of a large matrix without the prior knowledge of its eigenvalues.

\subsection*{Entropy formulae}
In quantum mechanics the density operator of a given state is a Hermitean, positive semidefinite operator. More specifically, focusing on real symmetric, positive semidefinite matrices~$\bm A\in\mathbb{R}^{m\times m}$, i.e., 
\[
\bm A=\bm A^\top,\qquad \bm v^\top\bm A\bm v\ge 0\quad\forall\bm v\in\mathbb{R}^m,
\]
the {\em entropy} of~$\bm A$ may be defined by
\begin{equation}\label{eq:spectrum}
\E(\bm A)=-\sum_{\lambda\in\s}\L(\lambda).
\end{equation}
Here, $\sigma(\bm A)\subset[0,\infty)$ signifies the spectrum of~$\bm A$, and~$\L$ is a continuous function on the real interval~$[0,\infty)$ which is given by
\begin{equation}\label{eq:L}
\L:\,[0,\infty)\to\mathbb{R},\qquad x\mapsto
\begin{cases}x\log(x)&\text{if }x>0\\0&\text{if }x=0\end{cases}.
\end{equation}

By the symmetry of~$\bm A$, the spectral theorem implies that~$\bm A$ can be diagonalized by means of an orthogonal matrix~$\bm Q\in\mathbb{R}^{m\times m}$, $\bm Q^{-1}=\bm Q^\top$, i.e. $\bm A=\bm Q\bm D\bm Q^\top$, where~$\bm D=\diag{\lambda_1,\lambda_2,\ldots,\lambda_m}$, and~$\lambda_1,\lambda_2,\ldots,\lambda_m\ge 0$ denote the eigenvalues of~$\bm A$. Then, we may define the matrix function
$
\L:\,\mathbb{R}^{m\times m}\to\mathbb{R}^{m\times m}
$
induced by~\eqref{eq:L} (and denoted with the same letter) in a standard way (see, e.g., \cite{HighamBook08}) by
$
\L(\bm A)=\bm Q\L(\bm D)\bm Q^\top$,
where
$
\L(\bm D)=\diag{\L(\lambda_1),\L(\lambda_2),\ldots,\L(\lambda_m)}$.
Furthermore, we see that
$
\E(\bm A)=-\tr(\L(\bm D))=-\tr(\bm Q^\top\L(\bm A)\bm Q)$.
Moreover, since the trace of a matrix is invariant with respect to similarity we arrive at
\begin{equation}\label{eq:entropy}
\E(\bm A)=-\tr(\L(\bm A)).
\end{equation}
At first glance, computing the entropy is possible by either applying~\eqref{eq:spectrum}, i.e., by computing the full spectrum of~$\bm A$, or by using formula~\eqref{eq:entropy} which involves the computation of the matrix logarithm. Evidently, for large matrices, both approaches are prone to be computationally unfeasible due to their high degree of complexity.

\subsection*{A new computational algorithm} The goal of this paper is to calculate the entropy of a matrix {\em without} the need of finding the eigenvalues of~$\bm A$ or the necessity of computing the matrix logarithm of~$\bm A$ explicitly. To this end, two key ingredients will be taken into account:
\begin{itemize}
\item The function~$\L$ will be approximated by a polynomial~$p$; in so doing the term~$\L(\bm A)\approx p(\bm A)$ can be expressed approximately as a sum of powers of~$\bm A$. This avoids the computation of the matrix logarithm.
\item Using the relations
$
\E(\bm A)=-\tr(\L(\bm A))\approx-\tr(p(\bm A))$,
the entropy of~$\bm A$ can be found approximately by computing the trace of $p(\bm A)$. This quantity, in turn, may be determined numerically by appropriately combining a Monte-Carlo procedure and a Clenshaw type scheme. In this way, the explicit computation of the matrix powers occurring in~$p(\bm A)$ can be circumvented.
\end{itemize}
Applying these ideas, we will obtain a low-complexity algorithm for the matrix entropy which is still able to generate accurate computational results. As a practical application we will consider a two-photon state entangled in frequency. More precisely, we will consider a large density matrix resulting from suitable discretization of the related continuous operator. We will use the new algorithm developed in this paper to demonstrate entanglement quantification by means of the entropy for a large real-valued density matrix~$\bm A\in\mathbb{R}^{m\times m}$, with a matrix size of the order~$m=\mathcal{O}(10^8)$.

The article is organized as follows: In Section~\ref{sc:MonteCarlo} we will recall a Monte-Carlo procedure proposed in~\cite{BaFaGo96} to compute the trace of a matrix function. Subsequently, a Chebyshev approximation polynomial of the function~$\L$ will be derived in Section~\ref{sc:p}, together with a sharp error estimate with respect to the supremum norm. Furthermore, Section~\ref{sc:entropy} contains the new algorithm and a probabilistic error analysis. Next, we present some numerical examples including a traditional finite element matrix and the above-mentioned quantum optics application in Section~\ref{sc:examples}. Finally, we draw some conclusions in Section~\ref{sc:conclusions}.

Throughout the paper, $\|\cdot\|_2$ denotes the Euclidean norm. Furthermore, $\tr(\cdot)$ signifies the trace of a matrix, i.e., the sum of its diagonal entries. We note the fact that~$\tr(\bm A)=\sum_{\lambda\in\s}\lambda$ for any~$\bm A\in\mathbb{R}^{m\times m}$. We also notice that, for any~$\gamma_0>0$, there holds that
\begin{equation}\label{eq:id}
\E(\bm A)=-\sum_{\lambda\in\s}\left[\gamma_0\L\left(\frac{\lambda}{\gamma_0}\right)+\log(\gamma_0)\lambda\right]=-\gamma_0\tr\left(\L\left(\gamma_0^{-1}\bm A\right)\right)-\log(\gamma_0)\tr(\bm A).
\end{equation}
The appealing property of this identity is that it allows to compute the entropy by means of the function~$\L$ as {\em restricted} to the interval~$[0,\nicefrac{\lambda}{\gamma_0}]$. In particular, since we approximate~$\L$ by a Chebyshev polynomial, the interval of approximation can be limited from~$[0,\max_{\lambda\in\s}\lambda]$ to the smaller interval~$[0,\gamma_0^{-1}\max_{\lambda\in\s}\lambda]$, if~$\gamma_0>1$.

\section{Monte-Carlo Approximation}\label{sc:MonteCarlo}

The following proposition, see, e.g.,~\cite{DongLiu94,Hutchinson89}, motivates a Monte-Carlo procedure for the computation of the trace of a symmetric matrix. Throughout the paper, we let~$\O=\{-1,+1\}$, and~$\O^m=\{-1,+1\}^m\subset\mathbb{R}^m$.

\begin{proposition}\label{pr:random}
Consider a symmetric matrix $\bm A\in\mathbb{R}^{m\times m}$ with~$\tr(\bm A)\neq 0$. Furthermore, let~$X$ be a random variable that takes values~$-1$ and~$1$ with probability~$\nicefrac12$ each. Moreover, let~$\bm\omega\in\O^m$ be a vector of~$m$ independent samples generated by~$X$. Then, 
$
\mathbb{E}(\bm\omega^\top\bm A\bm\omega)=\tr(\bm A)$,
where~$\mathbb{E}$ denotes the expected value.
\end{proposition}

From a practical point of view, this result allows for the computation of a numerical approximation of the trace of a symmetric matrix~$\bm A$ by taking the mean of a finite number~$N$ of sample computations,
\[
\tr(\bm A)\approx\frac{1}{N}\sum_{i=1}^N\bm\omega_i^\top\bm A\bm\omega_i,
\]
where~$\bm\omega_i\in\O^m$ are random vectors like defined above. Thence, recalling~\eqref{eq:id}, we find, for any~$\gamma_0>0$, that
\begin{equation}\label{eq:entropyN}
\E(\bm A)\approx-\frac{\gamma_0}{N}\sum_{i=1}^N \bm\omega_i^\top\L\left(\gamma_0^{-1}\bm A\right)\bm\omega_i-\log(\gamma_0)\tr(\bm A).
\end{equation}

In order to provide bounds for terms of the form~$\bm\omega^\top\L\left(\gamma_0^{-1}\bm A\right)\bm\omega$, for~$\bm \omega\in\O^m$, we note the following lemma. It is based on an elementary analysis of the graph of~$\L$.

\begin{lemma}\label{lm:q}
Let $x_0>0$. Then, there hold the estimates
\[
\min(\L(x_0),e^{-1}\sign(e^{-1}-x_0))\le\L(x)\le\max(0,\L(x_0)),
\]
for any~$x\in[0,x_0]$. Here, $\sign$ is the sign function (with~$\sign(0)=0$).
\end{lemma}

We continue by deriving some upper and lower bounds for~$\bm v^\top\L(\bm A)\bm v$. 

\begin{corollary}\label{co:bounds}
Let~$\bm A\in\mathbb{R}^{m\times m}$ be real symmetric, and positive semidefinite, and~$\s\subset[0,x_0\gamma_0]$, for some~$x_0,\gamma_0>0$. Furthermore, consider~$\bm v\in\O^m$. Then, the estimates
\[
m\gamma_0\min(\L(x_0),e^{-1}\sign(e^{-1}-x_0))\le
\gamma_0\bm v^\top\L(\gamma_0^{-1}\bm A)\bm v
\le m\gamma_0\max(0,\L(x_0))
\]
hold true for any~$\bm v\in\O^m$.
\end{corollary}

\begin{proof}
Let~$\bm v\in\O^m$. We choose an orthogonal matrix~$\bm Q$ that diagonalizes~$\bm A$, i.e. $\bm Q^\top\bm A\bm Q=\bm D=\diag{\lambda_1,\lambda_2,\ldots,\lambda_m}$. Then,
\[
\bm v^\top\L(\gamma_0^{-1}\bm A)\bm v
=(\bm Q^\top\bm v)^\top\L(\gamma^{-1}_0\bm D)\bm Q^\top\bm v
=\sum_{i=1}^m(\bm Q^\top\bm v)_i^2\L(\gamma_0^{-1}\lambda_i).
\]
Then, using the upper bound from Lemma~\ref{lm:q} and the identity~$\|\bm Q^\top\bm v\|_2=\|\bm v\|_2=\sqrt{m}$, it follows that
\begin{align*}
\bm v^\top\L(\gamma_0^{-1}\bm A)\bm v
&\le\max(0,\L(x_0))\sum_{i=1}^m(\bm Q^\top\bm v)_i^2=m\max(0,\L(x_0)).
\end{align*}
The proof of the lower bound is completely analogous.
\end{proof}

\begin{remark}\label{rm:gamma0}
Suppose that~$\lambda_{\max}>0$ is the maximal eigenvalue of~$\bm A$. Then, we choose $x_0,\gamma_0>0$ such that~$x_0\gamma_0=\lambda_{\max}$. Evidently, $\s\subset[0,x_0\gamma_0]$. Hence, by the previous Corollary~\ref{co:bounds}, we see that
\[
\lambda_{\max}mx_0^{-1}\min(\L(x_0),e^{-1}\sign(e^{-1}-x_0))\le
\gamma_0\bm v^\top\L(\gamma_0^{-1}\bm A)\bm v
\le \lambda_{\max}mx^{-1}_0\max(0,\L(x_0)).
\]
We may ask the question of how to choose~$x_0$ and~$\gamma_0$ such that the upper and lower bound in the above estimates are as close as possible to each other. In other words, we seek~$x_0^\mathrm{opt}>0$ such that the function
\begin{equation}\label{eq:d}
d(x_0)=x^{-1}_0\left[\max(0,\L(x_0))-\min(\L(x_0),e^{-1}\sign(e^{-1}-x_0))\right]
\end{equation}
is minimal. It turns out that this is the case for~$x^\mathrm{opt}_0=1$, i.e., $\gamma_0=\lambda_{\max}$; see Figure~\ref{fig:d}.
\end{remark}

\begin{psfrags}%
\psfragscanon%
%
\psfrag{s03}[t][t]{\color[rgb]{0,0,0}\setlength{\tabcolsep}{0pt}\begin{tabular}{c}$x_0$\end{tabular}}%
\psfrag{s04}[b][b]{\color[rgb]{0,0,0}\setlength{\tabcolsep}{0pt}\begin{tabular}{c}$d(x_0$)\end{tabular}}%
%
\psfrag{x01}[t][t]{0}%
\psfrag{x02}[t][t]{1}%
\psfrag{x03}[t][t]{2}%
\psfrag{x04}[t][t]{3}%
\psfrag{x05}[t][t]{4}%
\psfrag{x06}[t][t]{5}%
%
\psfrag{v01}[r][r]{$10^{0}$}%
%
\begin{figure}
\centering
\includegraphics[width=0.6\linewidth]{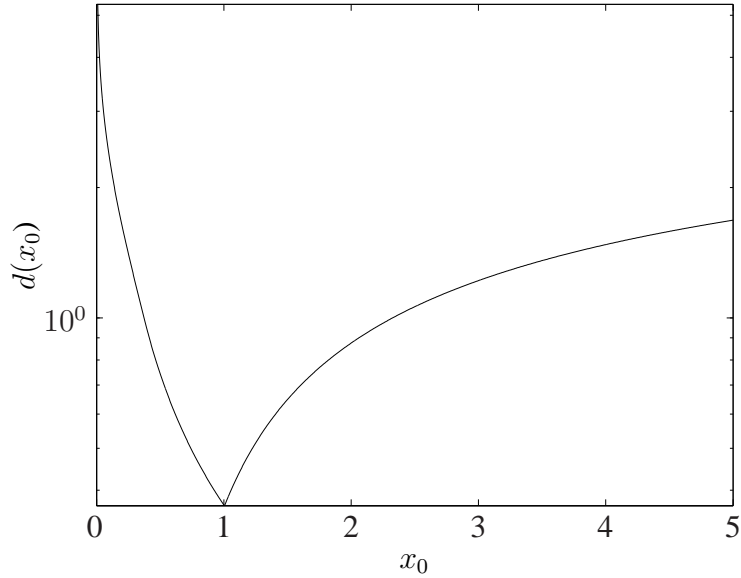}
\caption{Graph of the function~$d(x_0)$ from~\eqref{eq:d} in~$[0,5]$.}
\label{fig:d}
\end{figure}

\end{psfrags}%

\section{Chebyshev Approximation and Clenshaw's Algorithm}\label{sc:p}

Let us recall the Chebyshev polynomials of the first kind,
$
\widehat T_n(\widehat x)=\cos(n\arccos(\widehat x))$, $n\ge 0$,
on the reference interval~$\widehat I=[-1,1]$. They satisfy the three term recurrence relation
\[
\widehat T_{n+1}(\x)=2\x\widehat T_n(\x)-\widehat T_{n-1}(\widehat x),\qquad n\ge 1,
\]
with $\widehat T_0(\x)=1$, $\widehat T_1(\x)=\x$. These polynomials are orthogonal with respect to the inner product
\[
(f,g)=\int_{-1}^1\frac{f(\x)g(\x)}{\sqrt{1-\x^2}}\dx.
\]
More precisely,
\[
(\widehat T_m,\widehat T_n)=\begin{cases}
\pi&\text{if }m=n=0,\\
\displaystyle\frac{\delta_{mn}\pi}{2}&\text{if }m+n>0\\
\end{cases}
\]
where~$\delta_{mn}$ is Kronecker's delta; see, e.g., \cite{FoPa68}. Then, for~$x_0>0$, the affine transformation
\begin{equation}\label{eq:F}
F:\,[-1,1]\to[0,x_0],\qquad F(\x)=\frac{x_0}{2}(\x+1),\qquad F^{-1}(x)=\frac{2}{x_0}x-1,
\end{equation}
allows to define the Chebyshev polynomials~$\{T_n\}_{n\ge 0}$ on an interval~$I=[0,x_0]$:
\[
T_n=\widehat T_n\circ F^{-1},\qquad x\in I=[0,x_0].
\]

\begin{proposition}
Let~$x_0>0$, and~$n\in\mathbb{N}_0$. Then the function~$\L$ on the interval~$I=[0,x_0]$ is approximated by the polynomial function
\begin{equation}\label{eq:pn}
p_n(x)=\frac{a_0}{2}+\sum_{k=1}^na_kT_k(x),
\end{equation}
where the coefficients~$\{a_k\}_{k=0}^n$ are given by
\begin{align}
a_0&=x_0\left(\log\left(\frac{x_0}{4}\right)+1\right),\qquad
a_1=\frac{x_0}{4}\left(2\log\left(\frac{x_0}{4}\right)+3\right),\label{eq:a01}\\
a_k&=\frac{(-1)^kx_0}{k(k^2-1)},\qquad k\ge 2.\label{eq:ak}
\end{align}
Furthermore, for~$n\ge 1$, there holds the error estimate
\begin{equation}\label{eq:error}
\|\L-p_n\|_{\infty,(0,x_0)}\le\frac{x_0}{2n(n+1)},
\end{equation}
with~$\|\cdot\|_{\infty,(0,x_0)}$ denoting the supremum norm on~$I$.
\end{proposition}

\begin{proof}
We begin by defining the function
\[
\widehat\L=\L\circ F:\,\widehat x\mapsto\frac{x_0}{2}(\widehat x+1)\log\left(\frac{x_0}{2}(\widehat x+1)\right)
\]
on~$\widehat I=[-1,1]$. Using standard Fourier theory and affine scaling, the function~$\L$ can be represented by the infinite series
\[
\L(x)=\frac{a_0}{2}+\sum_{k=1}^\infty a_kT_k(x)
=\frac{a_0}{2}+\sum_{k=1}^\infty a_k(\widehat T_k\circ F^{-1})(x)=(\widehat{\L}\circ F^{-1})(x),
\]
with
\[
a_k=\frac{2}{\pi}\int_{-1}^1\frac{\widehat\L(\widehat x)\widehat T_k(\widehat x)}{\sqrt{1-\x^2}}\dx.
\]
Then, we define the polynomial~$p_n$ by truncation:
\[
p_n(x)=\frac{a_0}{2}+\sum_{k=1}^n a_kT_k(x).
\]
The coefficients~$\{a_k\}_{k=0}^n$ can be computed by employing the substitution~$\widehat x=\cos(t)$, cf.~\cite{Tr08}. This implies that
\[
a_k=\frac{2}{\pi}\int_0^{\pi}\widehat{\L}(\cos t)\cos(kt)\,\mathsf{d}t,\qquad k\ge 0.
\]
For~$k=0,1$ we find the formulas~\eqref{eq:a01} by direct calculation. In addition, noting the identity~$\cos t=\frac{1}{2}(e^{it}+e^{-it})$ and switching to complex variables~$z=e^{it}$, $\mathsf{d}z=iz\,\mathsf{d}t$, the formula~\eqref{eq:ak} follows from the residual theorem.

As for the error estimate, we notice that~$\|T_k\|_{\infty,(0,x_0)}=\|\widehat T_k\|_{\infty,(-1,1)}=1$. Then, for~$n\ge 1$,
\[
\|\L-p_n\|_{\infty,(0,x_0)}
=\left\|\sum_{k=n+1}^\infty a_kT_k\right\|_{\infty,(0,x_0)}
\le\sum_{k=n+1}^\infty|a_k|
=x_0\sum_{k=n+1}^\infty\frac{1}{k(k^2-1)}.
\]
Noticing the telescope sum
\[
\sum_{k=n+1}^\infty\frac{1}{k(k^2-1)}
=\frac12\sum_{k=n+1}^\infty\left(\frac{1}{k(k-1)}-\frac{1}{(k+1)k}\right)=\frac{1}{2n(n+1)}
\]
completes the proof.
\end{proof}

\begin{remark}
The polynomials~$p_n$, for~$n=0,1,2,3$, and~$x_0=3$ are shown, together with the moduli of the approximation errors, in Figure~\ref{fig:log}.
\end{remark}

\begin{figure}
\centering
\subfigure{
\providecommand\matlabtextA{\color[rgb]{0.000,0.000,0.000}\fontsize{5}{10}\selectfont\strut}%
\psfrag{017}[cl][cl]{\matlabtextA $x\log(x)$}%
\psfrag{018}[cl][cl]{\matlabtextA $n=3$}%
\psfrag{019}[cl][cl]{\matlabtextA $n=2$}%
\psfrag{020}[cl][cl]{\matlabtextA $n=1$}%
\psfrag{021}[cl][cl]{\matlabtextA $n=0$}%
\psfrag{022}[tc][tc]{\tiny{$x$}}%
%
%
%
\def\matlabfragNegXTick{\mathord{\makebox[0pt][r]{$-$}}}
\psfrag{000}[ct][ct]{\tiny{$0$}}%
\psfrag{001}[ct][ct]{\tiny{$0.5$}}%
\psfrag{002}[ct][ct]{\tiny{$1$}}%
\psfrag{003}[ct][ct]{\tiny{$1.5$}}%
\psfrag{004}[ct][ct]{\tiny{$2$}}%
\psfrag{005}[ct][ct]{\tiny{$2.5$}}%
\psfrag{006}[ct][ct]{\tiny{$3$}}%
%
%
%
\psfrag{007}[rc][rc]{\tiny{$-1$}}%
\psfrag{008}[rc][rc]{\tiny{$-0.5$}}%
\psfrag{009}[rc][rc]{\tiny{$0$}}%
\psfrag{010}[rc][rc]{\tiny{$0.5$}}%
\psfrag{011}[rc][rc]{\tiny{$1$}}%
\psfrag{012}[rc][rc]{\tiny{$1.5$}}%
\psfrag{013}[rc][rc]{\tiny{$2$}}%
\psfrag{014}[rc][rc]{\tiny{$2.5$}}%
\psfrag{015}[rc][rc]{\tiny{$3$}}%
\psfrag{016}[rc][rc]{\tiny{$3.5$}}%
%

\includegraphics[width=0.48\linewidth]{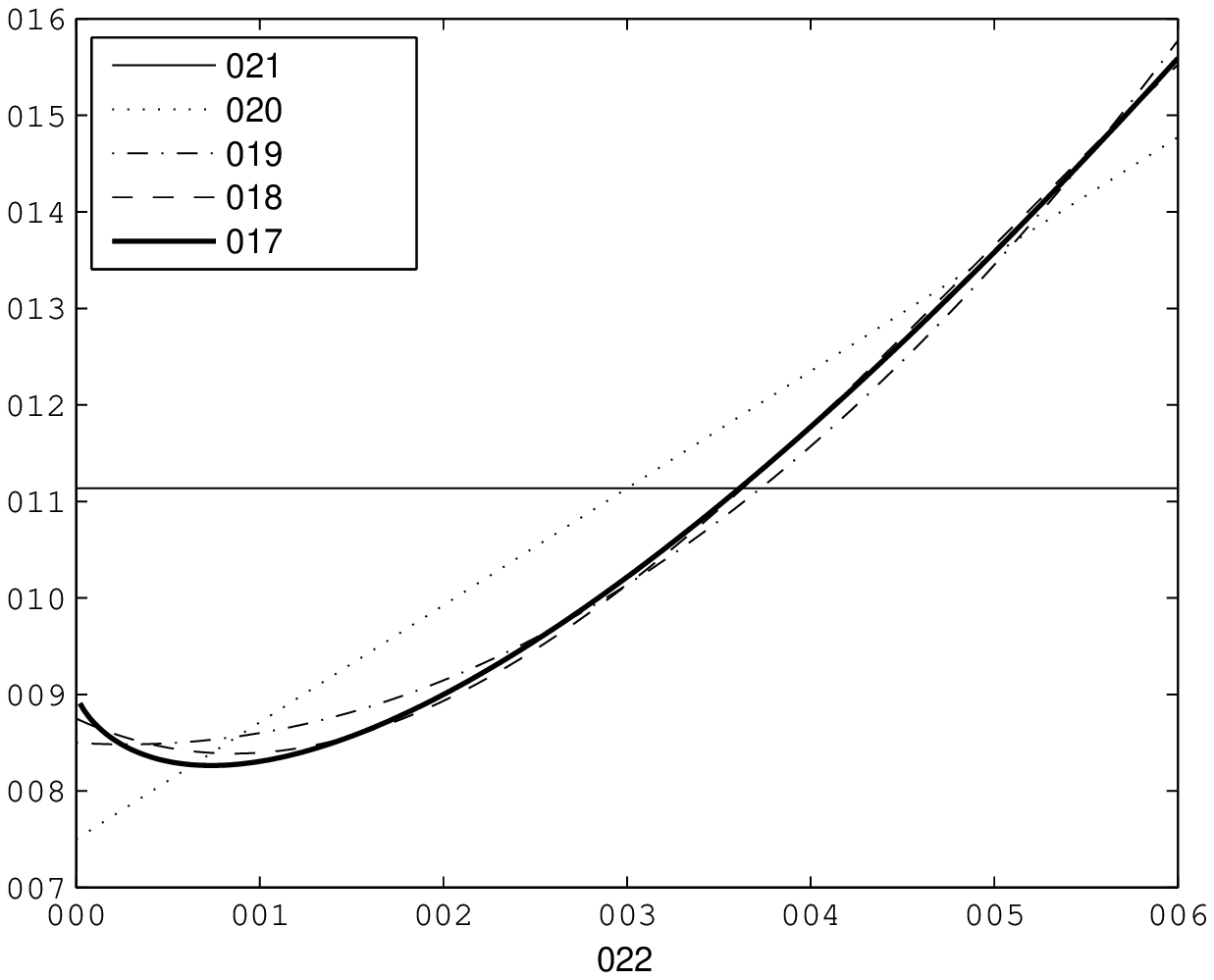}}
\subfigure{

\providecommand\matlabtextA{\color[rgb]{0.000,0.000,0.000}\fontsize{5}{10}\selectfont\strut}%
\psfrag{013}[cl][cl]{\matlabtextA $n=3$}%
\psfrag{014}[cl][cl]{\matlabtextA $n=2$}%
\psfrag{015}[cl][cl]{\matlabtextA $n=1$}%
\psfrag{016}[cl][cl]{\matlabtextA $n=0$}%
\psfrag{017}[tc][tc]{\tiny{$x$}}%
%
%
%
\def\matlabfragNegXTick{\mathord{\makebox[0pt][r]{$-$}}}
\psfrag{000}[ct][ct]{\tiny{$0$}}%
\psfrag{001}[ct][ct]{\tiny{$0.5$}}%
\psfrag{002}[ct][ct]{\tiny{$1$}}%
\psfrag{003}[ct][ct]{\tiny{$1.5$}}%
\psfrag{004}[ct][ct]{\tiny{$2$}}%
\psfrag{005}[ct][ct]{\tiny{$2.5$}}%
\psfrag{006}[ct][ct]{\tiny{$3$}}%
%
%
%
\psfrag{007}[rc][rc]{\tiny{$10^{-4}$}}%
\psfrag{008}[rc][rc]{\tiny{$10^{-3}$}}%
\psfrag{009}[rc][rc]{\tiny{$10^{-2}$}}%
\psfrag{010}[rc][rc]{\tiny{$10^{-1}$}}%
\psfrag{011}[rc][rc]{\tiny{$10^{0}$}}%
\psfrag{012}[rc][rc]{\tiny{$10^{1}$}}%
\includegraphics[width=0.48\linewidth]{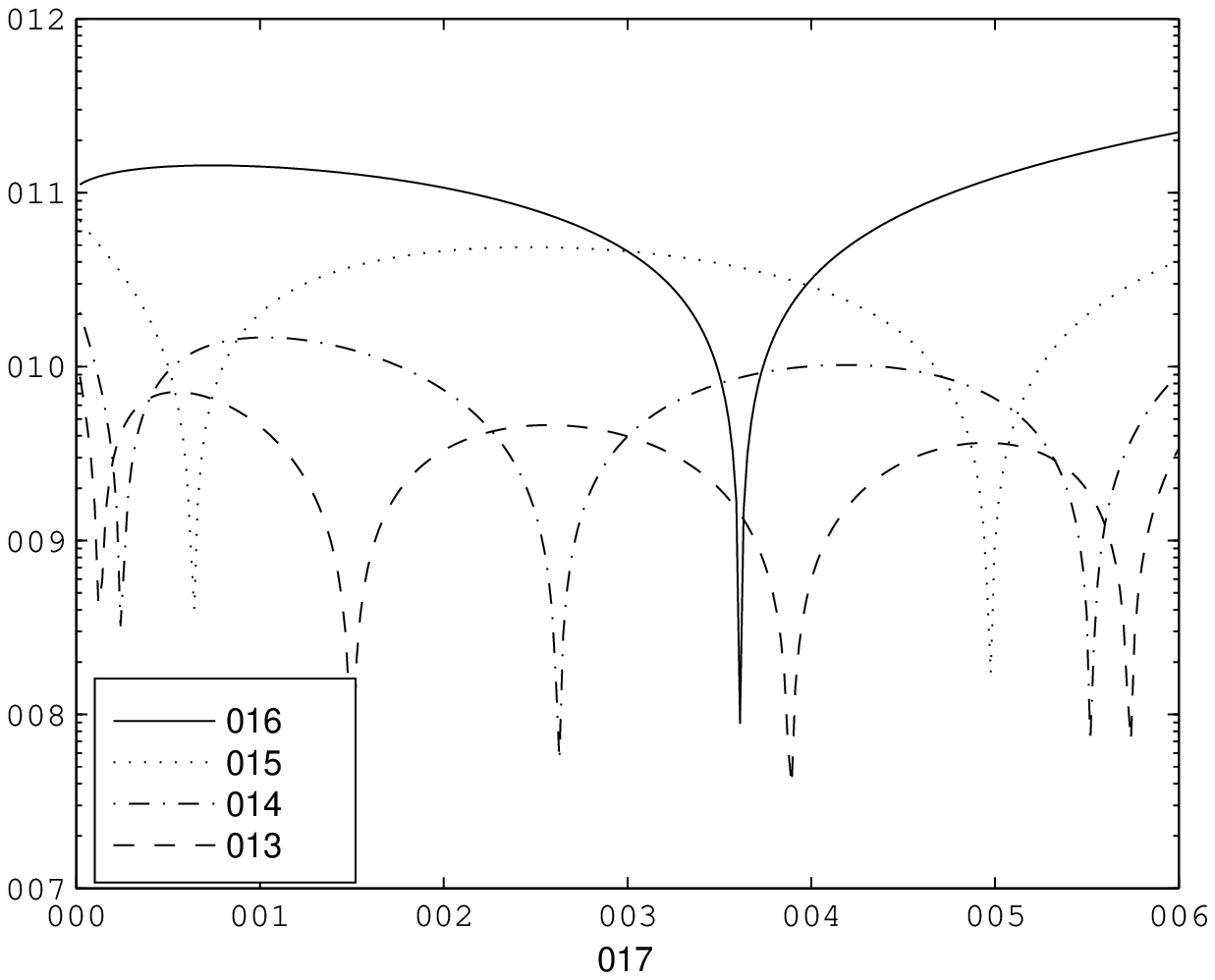}}

\caption{Polynomials~$p_n$, for~$n=0,1,2,3$, and~$x_0=3$. Left: Graphs of~$p_n$. Right: Approximation errors~$|\L(x)-p_n(x)|$.}
\label{fig:log}
\end{figure}

The above result implies the following estimates:

\begin{corollary}\label{co:Lpn}
Let~$\bm A\in\mathbb{R}^{m\times m}$ be symmetric and positive semidefinite, with~$\s\subset[0,x_0\gamma_0]$, for some~$x_0,\gamma_0>0$. Furthermore, consider~$\bm v\in\O^m$. Then, for~$n\ge 1$, it holds the bound
\[
|\bm v^\top(\L(\gamma_0^{-1}\bm A)-p_n(\gamma_0^{-1}\bm A))\bm v|\le\frac{mx_0}{2n(n+1)},
\]
where~$p_n$ is the Chebyshev approximation polynomial of the function~$\L$ from~\eqref{eq:pn}.
\end{corollary}

\begin{proof}
Let~$\bm Q\in\mathbb{R}^{m\times m}$ be orthogonal such that~$\bm Q^\top\bm A\bm Q=\diag{\lambda_1,\lambda_2,\ldots,\lambda_m}$. Then,
\begin{align*}
\bm v^\top(\L(\gamma_0^{-1}\bm A)-p_n(\gamma_0^{-1}\bm A))\bm v
&=(\bm Q^\top\bm v)^\top(\L(\gamma_0^{-1}\bm D)-p_n(\gamma_0^{-1}\bm D))\bm Q^\top\bm v\\
&=\sum_{i=1}^m (\bm Q^\top\bm v)_i^2(\L(\gamma_0^{-1}\lambda_i)-p_n(\gamma_0^{-1}\lambda_i)).
\end{align*}
Hence, since~$\|\bm Q^\top\bm v\|_2=\|\bm v\|_2=\sqrt{m}$, we arrive at
\[
\left|\bm v^\top(\L(\gamma_0^{-1}\bm A)-p_n(\gamma_0^{-1}\bm A))\bm v\right|
\le m\max_{1\le i\le m} |\L(\gamma_0^{-1}\lambda_i)-p_n(\gamma_0^{-1}\lambda_i)|.
\]
Finally, using the error estimate~\eqref{eq:error}, yields the desired result.
\end{proof}

In practical applications, Chebyshev series can be evaluated using the Clenshaw algorithm, see, e.g., \cite{Cl55,FoPa68}: Starting from functions
$
b_{n+2}(x)=b_{n+1}(x)=0$,
and applying affine scaling from~$[-1,1]$ to~$[0,x_0]$, see~\eqref{eq:F}, we define
\[
b_k(x) = a_k+2\left(\frac{2}{x_0}x-1\right)b_{k+1}(x)-b_{k+2}(x),\qquad k=n,n-1,\ldots,0,
\]
where~$\{a_k\}$ are the coefficients from~\eqref{eq:a01}--\eqref{eq:ak}. Then, it can be shown that~$p_n$ from~\eqref{eq:pn} can be represented in the form
\[
p_n(x)=\frac12(a_0+b_0(x)-b_2(x)).
\]

For a symmetric matrix~$\bm A\in\mathbb{R}^{m\times m}$ the polynomial~$\gamma_0p_n(\gamma_0^{-1}\bm A)$ can be computed in a similar way (indeed, this is possible since all terms involved consist of commuting sums of powers of~$\bm A$): Setting~$\bm B_{n+2}=\bm B_{n+1}=\bm 0\in\mathbb{R}^{m\times m}$ we define a finite sequence~$\{\bm B_k\}_{k=0}^{n+2}\subset\mathbb{R}^{m\times m}$ of matrices by the reverse recurrence relation
\[
\bm B_k=a_k\bm I+2\left(\frac{2}{x_0\gamma_0}\bm A-\bm I\right)\bm B_{k+1}-\bm B_{k+2},\qquad k=n,n-1,\ldots,0,
\]
where~$\bm I$ signifies the identity matrix in~$\mathbb{R}^{m\times m}$. Then,
\[
p_n(\gamma_0^{-1}\bm A)=\frac12\left(a_0\bm I+\bm B_0-\bm B_2\right).
\]
From the above relation we find for a vector~$\bm v\in\O^m$ that
\[
\bm B_k\bm v=a_k\bm v+\frac{4}{x_0\gamma_0}\bm A\bm B_{k+1}\bm v-2\bm B_{k+1}\bm v-\bm B_{k+2}\bm v,\qquad k=n,n-1,\ldots,0.
\]
Therefore, introducing the variable~$\bm y_k=\bm B_k\bm v$, we see that
\begin{align*}
\bm y_k&=a_k\bm v+\frac{4}{x_0\gamma_0}\bm A\bm y_{k+1}-2\bm y_{k+1}-\bm y_{k+2},
\end{align*}
for~$k=n,n-1,\ldots,0$, starting from~$\bm y_{n+2}=\bm y_{n+1}=\bm 0$. Finally, we obtain
\[
\bm v^\top p_n(\gamma_0^{-1}\bm A)\bm v=\frac12\left(ma_0+\bm v^\top(\bm y_0-\bm y_2)\right).
\]
Consequently, evaluating the above product essentially amounts to~$n$ matrix-vector multiplications.

\begin{algorithm}\label{alg:1}
Let~$\bm A\in\mathbb{R}^{m\times m}$ be a symmetric matrix, and~$\bm v\in\O^m$ a vector. Then, for any~$\gamma_0>0$, the quantity~$\gamma_0\bm v^\top p_n(\gamma_0^{-1}\bm A)\bm v$, where~$p_n$, $n\ge 1$, is the polynomial from~\eqref{eq:pn}, can be computed by means of the following procedure:
\begin{enumerate}
\item Set~$\bm y_{n+2}=\bm y_{n+1}=\bm 0\in\mathbb{R}^m$.
\item For $k=n,n-1,\ldots,0$ do
\[
\bm y_k=a_k\bm v+\frac{4}{x_0\gamma_0}\bm A\bm y_{k+1}-2\bm y_{k+1}-\bm y_{k+2}.
\]
\item Output~$\frac{\gamma_0}{2}\left(ma_0+\bm v^\top(\bm y_0-\bm y_2)\right)$.
\end{enumerate}
Here, $\{a_k\}_{k=0}^n$ are the coefficients from~\eqref{eq:a01}--\eqref{eq:ak}.
\end{algorithm}

\section{Computing the Entropy}\label{sc:entropy}

We now return to the idea of computing a numerical approximation of the entropy of a matrix by means of~\eqref{eq:entropyN}. In order to avoid the computation of the matrix logarithm, however, we will use the approximation~$\widetilde{\E}(\bm A)\approx\E(\bm A)$, where, for some~$\gamma_0>0$ to be specified later,
\begin{equation}\label{eq:entropypn}
\widetilde{\E}(\bm A)=-\frac{\gamma_0}{N}\sum_{i=1}^N \bm\omega_i^\top p_n\left(\gamma_0^{-1}\bm A\right)\bm\omega_i-\log(\gamma_0)\tr(\bm A).
\end{equation}
Here~$p_n$ is the approximation polynomial of degree~$n$ for~$\L$ from~\eqref{eq:pn}, and $\bm\omega_i\in\O^m$ are random vectors. Note that the expansion of a function in Chebyshev polynomials together with a stochastic evaluation of the trace according to Proposition~\ref{pr:random} also play an important role in the context of the kernel polynomial method (KPM); see~\cite{weisse2006}. 

In order to numerically evaluate (12), we propose the following basic algorithm:

\begin{algorithm}\label{alg:alg}
Let~$\bm A\in\mathbb{R}^{m\times m}$ be a real symmetric, positive semidefinite matrix, and~$n,N\in\mathbb{N}$. Furthermore, choose~$\gamma_0>0$. Then: 
\begin{enumerate}
\item Compute~$N$ random vectors~$\bm\omega_i\in\O^m$, $i=1,2,\ldots,N$, with entries~$\pm 1$ occurring with the same probability~$\nicefrac12$.
\item Determine the scalars~$\xi_i=\gamma_0\bm\omega_i^\top p_n(\gamma_0^{-1}\bm A)\bm\omega_i$ using Algorithm~\ref{alg:1}.
\item Output~$-\frac{1}{N}\sum_{i=1}^N\xi_i-\log(\gamma_0)\tr(\bm A)$.
\end{enumerate}
\end{algorithm}

The approximation provided by the above algorithm has two essential error sources: Firstly, the use of the Monte-Carlo approach~\eqref{eq:entropyN} brings about a certain randomness, and, secondly, replacing the function~$\L$ by~$p_n$ in~\eqref{eq:entropypn} leads to an approximation error. The latter point has been addressed already in Corollary~\ref{co:Lpn}. In order to deal with the issue of randomness, we provide a confidence interval analysis for the numerical approximation~\eqref{eq:entropypn} following the approach presented in~\cite{BaFaGo96}. To this end, we recall a special case of Hoeffding's inequality~\cite{Hoeffding63}:

\begin{proposition}
Let~$X_1, X_2,\ldots, X_N$ be independent random variables with zero means and bounded ranges~$\alpha_i^-\le X_i\le \alpha_i^+$, $i=1,2,\ldots, N$. Then, for any~$\eta>0$, there holds the probability bound
\[
P(|X_1+X_2+\ldots+X_N|\ge\eta)\le 2\exp\left(\frac{-2\eta^2}{\sum_{i=1}^N(\alpha_i^+-\alpha^-_i)^2}\right).
\]
\end{proposition}

In order to apply the previous result, we define, for~$i=1,2,\ldots, N$, the random variables
\begin{align*}
X_i&=-\gamma_0\bm\omega_i^\top\L(\gamma_0^{-1}\bm A)\bm\omega_i-\log(\gamma_0)\tr(\bm A)-\E(\bm A),
\end{align*}
where~$\bm\omega_i\in\O^m$ are random vectors with entries~$\pm 1$ appearing with equal probability of~$\nicefrac{1}{2}$. Using~\eqref{eq:id}, we conclude that
\begin{align*}
X_i&=\gamma_0\left[-\bm\omega_i^\top\L(\gamma_0^{-1}\bm A)\bm\omega_i+\tr(\L(\gamma_0^{-1}\bm A))\right].
\end{align*}
According to Proposition~\ref{pr:random}, we have that
$
\mathbb{E}(X_i)=0$, $i=1,2,\ldots,N$,
provided that $\tr(\L(\bm A))\neq 0$. Furthermore, we have that
\[
X_i=\gamma_0\left[-\bm\omega_i^\top p_n(\gamma_0^{-1}\bm A)\bm\omega_i+\tr(\L(\gamma_0^{-1}\bm A))\right]+\gamma_0\bm\omega_i^\top\left[-\L(\gamma_0^{-1}\bm A)+p_n(\gamma_0^{-1}\bm A)\right]\bm\omega_i,
\]
and thus, with the aid of Corollary~\ref{co:Lpn},
$
\alpha_i^-\le X_i\le \alpha_i^+$,
with
\[
\alpha_i^\pm=\gamma_0\left[-\bm\omega_i^\top p_n(\gamma_0^{-1}\bm A)\bm\omega_i+\tr(\L(\gamma_0^{-1}\bm A))\right]\pm\frac{mx_0\gamma_0}{2n(n+1)},
\]
for $i=1,2,\ldots, N$. Then, setting
\begin{equation*}
\alpha_{\min}=\min_{1\le i\le N}\alpha_i^-,\qquad 
\alpha_{\max}=\max_{1\le i\le N}\alpha_i^+,
\end{equation*}
and
\begin{equation}\label{eq:delta}
\delta=\alpha_{\max}-\alpha_{\min}=\max_{1\le i\le N}{-\gamma_0\bm\omega_i^\top p_n(\gamma_0^{-1}\bm A)\bm\omega_i}-\min_{1\le i\le N}{-\gamma_0\bm\omega_i^\top p_n(\gamma_0^{-1}\bm A)\bm\omega_i}+\frac{mx_0\gamma_0}{n(n+1)},
\end{equation}
we obtain the uniform bounds $\alpha_{\min}\le X_i \le\alpha_{\max}$, $1\le i\le N$,
and hence, by Hoeffding's inequality, Proposition~\ref{pr:random}, we find, for any~$\eta>0$, the probability estimate
\begin{equation}\label{eq:aux1}
\begin{split}
P\left(\left|\frac{1}{N}\sum_{i=1}^NX_i\right|\ge\frac{\eta}{N}\right)&=
P\left(\left|-\frac{\gamma_0}{N}\left[\sum_{i=1}^N\bm\omega_i^\top\L(\gamma_0^{-1}\bm A)\bm\omega_i\right]-\log(\gamma_0)\tr(\bm A)-\E(\bm A)\right|\ge\frac{\eta}{N}\right)\\
&\le 
2\exp\left(-2N\delta^{-2}\left(\nicefrac{\eta}{N}\right)^2\right).
\end{split}
\end{equation}
Using the approximation~\eqref{eq:entropypn}, and recalling again Corollary~\ref{co:Lpn} results in
\begin{align*}
\left|\E(\bm A)-\widetilde{\E}(\bm A)\right|
&=\left|-\frac{\gamma_0}{N}\left[\sum_{i=1}^N \bm\omega_i^\top p_n\left(\gamma_0^{-1}\bm A\right)\bm\omega_i\right]-\log(\gamma_0)\tr(\bm A)-\E(\bm A)\right|\\
&\le\left|-\frac{\gamma_0}{N}\sum_{i=1}^N \bm\omega_i^\top \left(p_n\left(\gamma_0^{-1}\bm A\right)-\L\left(\gamma_0^{-1}\bm A\right)\right)\bm\omega_i\right|\\
&\quad+\left|-\frac{\gamma_0}{N}\left[\sum_{i=1}^N\bm\omega_i^\top\L(\gamma_0^{-1}\bm A)\bm\omega_i\right]-\log(\gamma_0)\tr(\bm A)-\E(\bm A)\right|\\
&\le \frac{mx_0\gamma_0}{2n(n+1)}+\left|-\frac{\gamma_0}{N}\left[\sum_{i=1}^N\bm\omega_i^\top\L(\gamma_0^{-1}\bm A)\bm\omega_i\right]-\log(\gamma_0)\tr(\bm A)-\E(\bm A)\right|.
\end{align*}
Thus, with~\eqref{eq:aux1}, we obtain
\begin{align*}
P&\left(\left|\E(\bm A)-\widetilde{\E}(\bm A)\right|\ge\frac{\eta}{N}+\frac{mx_0\gamma_0}{2n(n+1)}\right)\\
&\le P\left(\left|-\frac{\gamma_0}{N}\left[\sum_{i=1}^N\bm\omega_i^\top\L(\gamma_0^{-1}\bm A)\bm\omega_i\right]-\log(\gamma_0)\tr(\bm A)-\E(\bm A)\right|\ge\frac{\eta}{N}\right)\\
&\le 2\exp\left(-2N\delta^{-2}\left(\nicefrac{\eta}{N}\right)^2\right),
\end{align*}
and therefore,
\begin{equation}\label{eq:aux2}
P\left(\left|\E(\bm A)-\widetilde{\E}(\bm A)\right|<\frac{\eta}{N}+\frac{mx_0\gamma_0}{2n(n+1)}\right)
> 
1-2\exp\left(-2N\delta^{-2}\left(\nicefrac{\eta}{N}\right)^2\right).
\end{equation}
Now, fixing an error tolerance~$\tau>0$, we select~$\nicefrac{\eta}{N}>0$ such that
\[
\frac{\eta}{N}=\tau-\frac{mx_0\gamma_0}{2n(n+1)}.
\]
Hence,
\[
P\left(\left|\E(\bm A)-\widetilde{\E}(\bm A)\right|<\tau\right)
> 
1-2\exp\left(-2N\delta^{-2}\left(\tau-\frac{mx_0\gamma_0}{2n(n+1)}\right)^2\right).
\]

We thus have proved the following result:

\begin{theorem}\label{th:main1}
Consider a real symmetric, positive semidefinite matrix~$\bm A\in\mathbb{R}^{m\times m}$, and constants $\gamma_0,x_0>0$ such that~$\s\subset[0,\gamma_0x_0]$. Moreover, let~$\tau>0$ be a prescribed error tolerance, and~$n\in\mathbb{N}$ a polynomial degree such that
\begin{equation}\label{eq:n}
\tau>\frac{mx_0\gamma_0}{2n(n+1)}.
\end{equation}
Then, computing~$N\in\mathbb{N}$ sample vectors $\bm\omega_i\in\O^m$, $i=1,2,\ldots,N$,
with entries~$\pm 1$ of equal probability~$\nicefrac12$, the output of Algorithm~\ref{alg:alg}, denoted by~$\widetilde\E(\bm A)$, satisfies
\begin{equation}\label{eq:errorbound}
\left|\E(\bm A)-\widetilde\E(\bm A)\right|<\tau,
\end{equation}
with probability at least
\begin{equation}\label{eq:p}
p=1-2\exp\left(-2N\delta^{-2}\left(\tau-\frac{mx_0\gamma_0}{2n(n+1)}\right)^2\right).
\end{equation}
Here, $\delta=\alpha_{\max}-\alpha_{\min}$ is defined in~\eqref{eq:delta}.
\end{theorem}

\begin{remark}
The above theorem shows that, in order to achieve a certain prescribed accuracy~$\tau$ in the computations, the polynomial degree~$n$ of~$p_n$ from~\eqref{eq:pn} needs to be sufficiently large in accordance with~\eqref{eq:n}. In addition, we see that the probability~$p$ of satisfying the error estimate~\eqref{eq:errorbound} can be increased by adding more samples in the Monte-Carlo approach. In addition, from~\eqref{eq:p}, it follows that
\[
N=\frac12\delta^2\left(\tau-\frac{mx_0\gamma_0}{2n(n+1)}\right)^{-2}\log\left(\frac{2}{1-p}\right).
\]
Therefore, noticing that~$\delta=\mathcal{O}(m)$ (cf.~Corollary~\ref{co:bounds}) may imply that the theorem could require~$N$ to be unfeasibly large. Consequently, again following~\cite[Section 4.2]{BaFaGo96}, it may often be more practical to fix the number~$N$ of samples, or in this paper, a polynomial degree~$n$ beforehand, and to provide an error bound for a given probability~$p$. Indeed, with~$p$ from~\eqref{eq:p} we solve for~$\tau$ to arrive at
\begin{equation}\label{eq:tau}
\tau=\frac{mx_0\gamma_0}{2n(n+1)}+\delta\sqrt{\frac{1}{2N}\log\left(\frac{2}{1-p}\right)},
\end{equation}
where we have obeyed~\eqref{eq:n} in choosing the sign in front of the square root. We notice that~$\tau$ is a sum of two independent error contributions. Thus, for given polynomial degree~$n$ it is reasonable to choose the number of samples~$N$ such that
\[
\frac{mx_0\gamma_0}{2n(n+1)}=\delta\sqrt{\frac{1}{2N}\log\left(\frac{2}{1-p}\right)},
\]
i.e.,
\begin{equation}\label{eq:N}
N=\frac{2n^2(n+1)^2\delta^2\log\left(\frac{2}{1-p}\right)}{m^2x_0^2\gamma_0^2}.
\end{equation}
This observation leads to Algorithm~\ref{alg:algfinal} below.
\end{remark}

\begin{algorithm}\label{alg:algfinal}
Let~$\bm A\in\mathbb{R}^{m\times m}$ be a real symmetric, positive semidefinite matrix, and~$n\in\mathbb{N}$ a prescribed polynomial degree. Furthermore, choose~$x_0,\gamma_0>0$ with~$\s\subset[0,x_0\gamma_0]$, and a probability~$p\in(0,1)$.
Then: 
\begin{enumerate}
\item Set $i=0$, $N=1$, $\xi_{\max}=-\infty$, $\xi_{\min}=\infty$.
\item While $i<N$ do
\begin{itemize}
\item $i=i+1$.
\item Find a random vector~$\bm\omega_i\in\O^m$ with entries~$\pm 1$ occurring with the same probability~$\nicefrac12$.
\item Determine the scalar~$\xi_i=\gamma_0\bm\omega_i^\top p_n(\gamma_0^{-1}\bm A)\bm\omega_i$ using Algorithm~\ref{alg:1}.
\item Compute
$\xi_{\min}=\min(\xi_{\min},\xi_i)$ and
$\xi_{\max}=\max(\xi_{\max},\xi_i)$.
\item Find
\[
\delta=\xi_{\max}-\xi_{\min}+\frac{mx_0\gamma_0}{n(n+1)},
\] 
and $N$ from~\eqref{eq:N}.
\end{itemize}
End do.
\item Output the approximate entropy
\[
\widetilde{\E}(\bm A)=-\frac{1}{N}\sum_{i=1}^N\xi_i-\log(\gamma_0)\tr(\bm A),
\]
and the error tolerance from~\eqref{eq:tau}.
\end{enumerate}
\end{algorithm}

\begin{remark} In accordance with Remark~\ref{rm:gamma0} it is sensible to choose~$x_0=1$ and~$\gamma_0=\widetilde\lambda_{\max}$, where~$\widetilde\lambda_{\max}>0$ is an upper bound on the spectrum~$\s$. In practice, $\widetilde\lambda_{\max}$ needs to be determined by suitable algorithms; for example, for sparse matrices with comparatively small off-diagonal elements, the Gerschgorin circle theorem~\cite{Gerschgorin31} could be applied; more generally, there are various numerical methods for finding the maximal eigenvalue of a large (and possibly sparse) matrix including, in particular, iterative schemes such as the Arnoldi algorithm (we refer to~\cite{saad2011} for details on different methods).
\end{remark}

\section{Examples}\label{sc:examples}

We shall now illustrate the method developed in this paper by means of two examples.

\subsection{A Finite Element Matrix}

We consider the classical stiffness matrix
\[
\bm A=\begin{pmatrix}
2& -1& 0 &\cdots& 0\\
-1& 2& -1&\ddots&\vdots\\
 0 & \ddots  & \ddots &\ddots &0\\
\vdots&\ddots &-1& 2& -1\\
0&\cdots&0 &-1& 2\\
\end{pmatrix}\in\mathbb{R}^{m\times m},
\]
which appears in the discretization of the one-dimensional boundary value problem
\[
-u''(x)=f(x),\quad x\in(0,1),\qquad u(0)=u(1)=0,
\]
by uniform linear finite element; see, e.g., \cite[Chap.~1]{Jo09}. It can be shown that the eigenvalues of~$\bm A$ are given by
\[
\lambda_i=4\sin^2\left(\frac{i\pi}{2m+2}\right),\qquad 1\le i\le m,
\]
and therefore,
\begin{align*}
\E(\bm A)&=-4\sum_{i=1}^m\sin^2\left(\frac{i\pi}{2m+2}\right)\log\left(4\sin^2\left(\frac{i\pi}{2m+2}\right)\right)\\
&\approx -\frac{4(2m+2)}{\pi}\int_0^{\nicefrac{\pi}{2}}\sin^2(x)\log\left(4\sin^2(x)\right)\,\mathsf{d}x.
\end{align*}
Now, using the fact that~$\int_0^{\nicefrac{\pi}{2}}\sin^2(x)\log\left(4\sin^2(x)\right)\,\mathsf{d}x=\nicefrac{\pi}{4}$, we see that~$\E(\bm A)\approx -2m$. In particular, the entropy decreases asymptotically linearly as~$m\to\infty$.

In Table~\ref{tb:fem} we present numerical results for a prescribed probability~$p=0.95$ and several polynomial degrees~$n$. The latter quantity has been chosen `by hand' with moderate growth as the matrix size~$m$ is increasing. We clearly see that the algorithm generates quite accurate results already for a low number of samples. Indeed, the relative errors are (except for~$m=10$) below 1\%, and the computed errors based on~\eqref{eq:tau} are very reasonable as compared to the magnitude of the exact entropy. 


\begin{table}
\begin{center}
\begin{tabular}{c|c|c|c|c|c|c}
\hline
matrix & polyn. & number of  & exact& abs.~err. & rel.~err. & comput.\\
size~$m$& deg.~$n$ & samples~$N$& entropy& &&err.\\\hline\hline
    10  &          2     &      21   &    $-0.19232\cdot 10^2$    &  0.2127   &    1.1057\%    &   $0.66293\cdot 10^1$\\
           50   &         3     &      33  &     $-0.99228\cdot10^2$  &    0.7396  &     0.7453\%  &     $0.16587\cdot 10^2$\\
          100   &         3     &      37  &     $-0.19923\cdot10^3$  &   0.0749  &   0.0376\%  &     $0.33149\cdot 10^2$\\
          500  &          4     &      18  &     $-0.99923\cdot10^3$  &     2.2705  &    0.2272\%  &     $0.98905\cdot 10^2$\\
         1000  &          6     &      35  &     $-0.19992\cdot10^4$  &     1.1163  &   0.0558\%  &     $0.94559\cdot 10^2$\\
         5000  &          8      &     15  &     $-0.99992\cdot10^4$  &     7.4981  &   0.0750\%  &     $0.27554\cdot 10^3$\\
 \hline
\end{tabular}
\end{center}
\caption{Entropy of a finite element matrix for various sizes~$m$ and $p=0.95$.}
\label{tb:fem}
\end{table}

\subsection{An Application in Quantum Optics}\label{sc:qo}
Entangled photons have become a widely used non-classical light source to investigate fundamental aspects of entanglement \cite{RevModPhys.71.S288,Genovese2005}. Their unique properties have further paved the way to potentially practical applications in quantum communication and quantum computing \cite{RevModPhys.79.135,Gisin2007}. In recent years spontaneous parametric down-conversion (SPDC) has become the standard procedure to generate entangled photon states. SPDC occurs when a noncentrosymmetric crystal is pumped by a laser beam strong enough to induce nonlinear interactions. In this case, a pump photon with angular frequency $\omega_p$ may be annihilated and two new photons of lower frequencies $\omega_i$ and $\omega_s$, denoted as the {\em idler} and the {\em signal}, are created. Energy conservation demands $\omega_i+\omega_s=\omega_p$. If the experimental configuration of the three involved photons is further restricted to the case where they propagate collinearly, the resulting two-photon state, given by,
\begin{equation}\label{dc_state}
\vert \Psi \rangle=\vert 0\rangle + \int \mathsf{d}\omega_i \int \mathsf{d}\omega_s \; f(\omega_i,\omega_s)\; \hat{a}_i^\dagger(\omega_i)\hat{a}_s^\dagger(\omega_s) \; \vert 0\rangle,
\end{equation}
describes entanglement in the frequency domain \cite{keller_rubin97pra}. We consider here identically polarized photon states created by the action of $\hat{a}_j^\dagger(\omega_j)$, $j\in\{i,s\}$, on the combined vacuum state $\vert 0\rangle \doteq\vert 0\rangle_i\otimes \vert 0\rangle_s$. The state in (\ref{dc_state}) is an entangled state if the joint spectral amplitude
\begin{equation}\label{dc_f}
f(\omega_i,\omega_s)\propto\exp\left(-\frac{(\omega_i+\omega_s-\omega_{cp})^2\tau_p^2}{8\log(2)}\right)\mbox{sinc}\left(\frac{\Delta k(\omega_i,\omega_s)L}{2}\right)
\end{equation}
cannot be separated into a product $f(\omega_i,\omega_s)=g(\omega_i)h(\omega_s)$. The pump pulse with center frequency $\omega_{cp}$ is represented by the exponential term in (\ref{dc_f}) and its duration is given by $\tau_p$. The parameter $L$ denotes the length of the crystal. The efficiency of the SPDC process is dominated by $\Delta k(\omega_i,\omega_s)=k_i(\omega_i)+k_s(\omega_s)-k_p(\omega_i+\omega_s)+\frac{2\pi}{G}$, where $k_j(\omega)$ is the frequency-dependent propagation constant of a periodically poled crystal with poling period $G$. Using the corresponding Schmidt decomposition, the amount of entanglement in (\ref{dc_state}) can now be quantified by the entropy (\ref{eq:id}) of either idler or signal subsystem \cite{bennet_etalpra96}. The state of each subsystem is described by its corresponding continuous density matrix, given by,
\begin{align*}
A_i(\omega,\omega')&=\int \mathsf{d}\omega_s\; f(\omega,\omega_s)f^*(\omega',\omega_s), &
A_s(\omega,\omega')&=\int \mathsf{d}\omega_i\; f(\omega_i,\omega)f^*(\omega_i,\omega').
\end{align*}
Due to the symmetry of $f(\omega_i,\omega_s)$ in (\ref{dc_f}) we define $A_i(\omega,\omega')=A_s(\omega,\omega')\doteq A(\omega,\omega')\in\mathbb{R}$. In order to calculate (\ref{eq:id}), the continuous function $A(\omega,\omega')$ has to be discretized on a lattice, i.e.~$A(\omega,\omega')\rightarrow \bm A\in\mathbb{R}^{m\times m}$, with $\bm A=\bm A^\top$. Since $\bm A$ is the density operator of a physical state, its eigenvalues are further distributed such that $\lambda \ge 0$ for all $\lambda\in\s$. For short pump pulses, where $\tau_p$ is of the order of fs, the exact $\E(\bm A)$ can be calculated by means of $\s$ since only small grid sizes ($m\approx800$) are needed to sufficiently resolve $A(\omega,\omega')$ \cite{law_etal00prl}. Unfortunately, the grid sizes required to discretize $A(\omega,\omega')$ for long pump pulses, e.g., for~$\tau_p$ on a timescale of ns, are very large since in this case $f(\omega_i,\omega_s)$ is dominated by a narrow Gaussian function. Diagonalization of $\bm A$ is then practically unfeasible. However, a numerical approximation of the entropy according to Algorithm 4.5 is still possible. We have calculated the entropy $\widetilde{\E}(\bm A/\tr(\bm A))$ for a pump pulse duration of $\tau_p=88.3$ ns and a $L=11.5$ mm long potassium titanyl phosphate crystal with $G=9.014$ $\mu$m. In the frequency domain, this specific choice of $\tau_p$ corresponds to a pump pulse with a narrow spectral bandwidth of 5 MHz. Notice, that the normalization~$\bm A/\tr(\bm A)$ results from the fact that a physical state needs to be normalized; indeed, since~$\tr(\bm A)=\sum_{\lambda\in\sigma(\bm A)}\lambda$ we have that~$\tr(\bm A/\tr(\bm A))=1$. In order to save memory space we made use of the fact that only a small amount of entries in $\bm A$ are significantly nonzero which allows to store the matrix in sparse format. This procedure results in a band matrix with $m=0.72\cdot 10^8$ and 37 diagonals. Figure \ref{fig:entropy} shows convergence of $\widetilde{\E}(\bm A/\tr(\bm A))$ and the error tolerance from~\eqref{eq:tau}. For polynomial degree $n=20$ and error probability $p=0.95$ we obtain $\widetilde{\E}(\bm A/\tr(\bm A))=14.969\pm0.128$. Up to a polynomial degree of $n=20$, the discretization error for $m=0.72\cdot 10^8$ is still smaller than the computational error~$\tau$. For all $n$ the number $N$ of sample vectors is $N=8$ and with Gerschgorin's theorem one obtains $\gamma_0=\tilde{\lambda}_{max}=7.57\cdot 10^{-7}$. Our computations were performed in {\sc Matlab}\footnote{{\sc Matlab} is a trademark of The MathWorks, inc.} on a 12 core Intel Xeon X5650 (2.66 GHz) processor with 96 GB RAM. It is remarkable that for a matrix size $m=0.72\cdot 10^8$ the computational time for the entropy only took about 25 minutes. This clearly underlines the high efficiency of Algorithm 4.5 for this example. In the case of a \textit{maximally} entangled, discrete bipartite system of finite dimension~$m^2$, the entropy increases according to~$\widetilde{\E}=\mathcal{O}(\log(m))$. Due to~$\tau_p$ being of the order of ns, the state under consideration exhibits a very high degree of entanglement and is therefore almost equivalent to a maximally entangled system with~$m\approx\exp(14.969)$.

\begin{psfrags}%
\psfragscanon%
%
\psfrag{012}[tc][tc]{\color[rgb]{0,0,0}\setlength{\tabcolsep}{0pt}\begin{tabular}{c}$n$\end{tabular}}%
\psfrag{011}[bc][bc]{\color[rgb]{0,0,0}\setlength{\tabcolsep}{0pt}\begin{tabular}{c}$\widetilde{\mathfrak{E}}(\bm A/\tr(\bm A))$\end{tabular}}%
%
\psfrag{000}[ct][ct]{$0$}%
\psfrag{001}[ct][ct]{$5$}%
\psfrag{002}[ct][ct]{$10$}%
\psfrag{003}[ct][ct]{$15$}%
\psfrag{004}[ct][ct]{$20$}%
%
\psfrag{005}[rc][rc]{$0$}%
\psfrag{006}[rc][rc]{$10$}%
\psfrag{007}[rc][rc]{$20$}%
\psfrag{008}[rc][rc]{$30$}%
\psfrag{009}[rc][rc]{$40$}%
\psfrag{010}[rc][rc]{$50$}%

%
\begin{figure}
\centering
\includegraphics[width=0.63\linewidth]{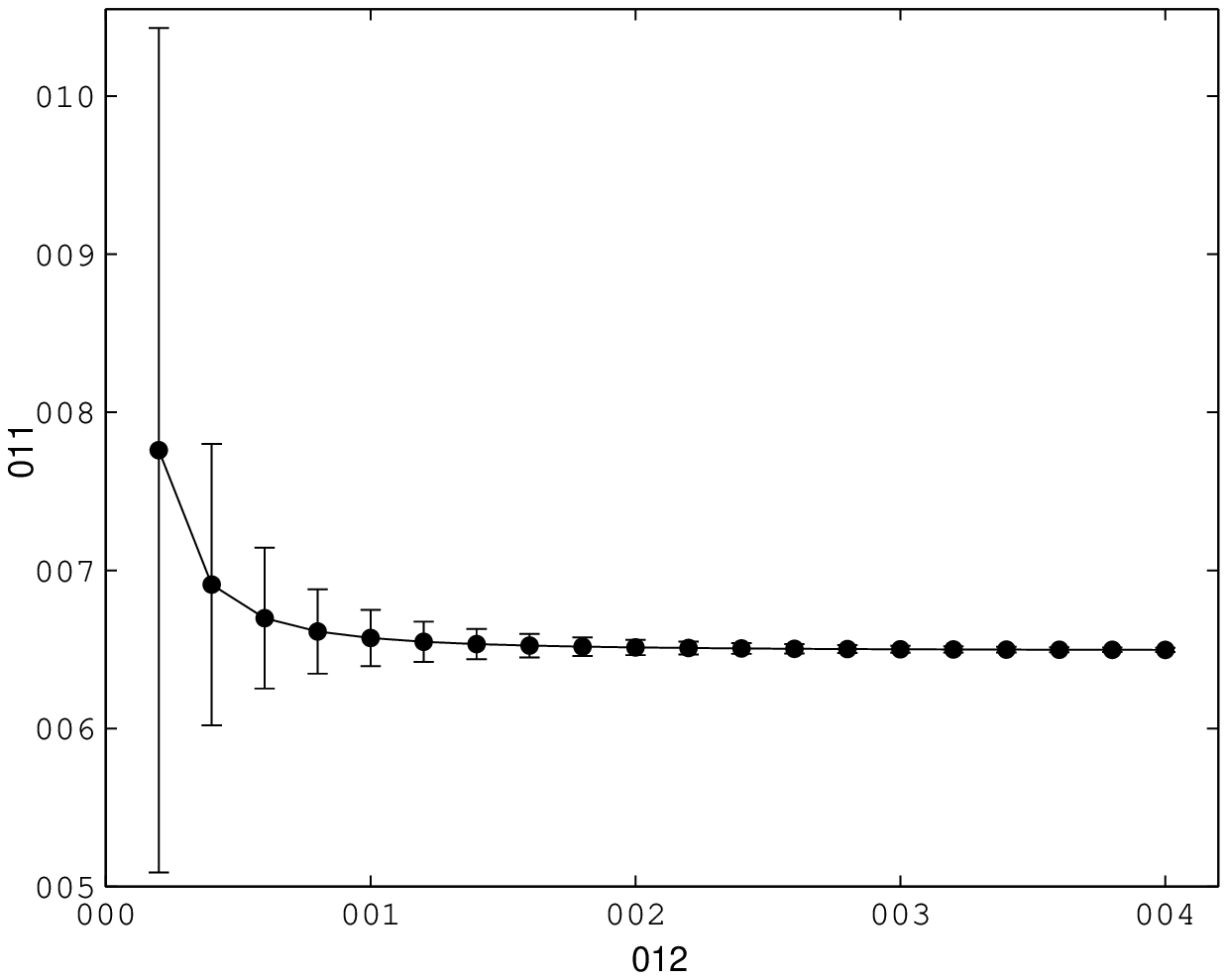}%
\caption{Quantum optics application, for matrix size~$m=0.72\cdot 10^8$: Approximate entropy $\widetilde{\E}(\bm A/\tr(\bm A))$ based on Algorithm 4.5 with~$p=0.95$ for various polynomial degrees $n$; the vertical bars indicate the computational error ranges according to~\eqref{eq:tau}.}
\label{fig:entropy}
\end{figure}
\end{psfrags}%

\section{Conclusions}\label{sc:conclusions}

In this article, we have derived a new algorithm for the computation of the entropy of a large real symmetric, positive semidefinite matrix. The proposed procedure does neither require the computation of the spectrum nor of the matrix logarithm. Indeed, it is based on the following two main ideas:
\begin{itemize}
\item Approximation of the `entropy function'~$\L$ by a reasonably accurate Chebyshev polynomial.
\item Computation of the entropy by combining a Monte-Carlo type sampling procedure and a Clenshaw algorithm for matrix polynomials.
\end{itemize}
The new algorithm is parallelizable and straightforward to implement. It was tested for a classical finite element matrix as well as for a large discretization matrix originating from a quantum optics application. In both cases, our algorithm is able to achieve accurate results in a very efficient way.

An interesting extension of this research constitutes the computation of the matrix entropy in the context of {\em complex} discrete Hermitean operators. Here, an important ingredient is the appropriate redefinition of the function~$\L$ for complex input values and the corresponding approximation by polynomial functions for both the real as well as the imaginary part.

The algorithm in this work can be extended to evaluate the relative entropy  $\E(\sigma||\rho)=\tr(\sigma\log(\sigma))-\tr(\sigma\log(\rho))$ between two density matrices $\sigma$ and $\rho$. This quantity determines the relative entropy of entanglement $E(\sigma)=\min_{\rho\in D}\E(\sigma||\rho)$, where $D$ is the set of all disentangled states. A numerical algorithm for the computation of the relative entropy of entanglement was proposed for low dimensions \cite{Zinchenko2010}, however, no algorithm is yet available for large systems.

\section*{Acknowledgements}
BB and AS were supported by the grant PP00P2\_133596 funded by the Swiss National Science Foundation. TW acknowledges funding by the Swiss National Science Foundation under grant~200020\_144442.

\bibliographystyle{unsrt} 

\begin{thebibliography}{10}

\bibitem{Shannon1948}
C.~E. Shannon.
\newblock {A mathematical theory of communication}.
\newblock {\em The Bell System Technical Journal}, 27(July 1928):379--423,
  1948.

\bibitem{Bekenstein1973}
J.~D. Bekenstein.
\newblock {Black holes and entropy}.
\newblock {\em Phys. Rev. D}, 7(8):2333, 1973.

\bibitem{Wehrl1978}
A.~Wehrl.
\newblock {General properties of entropy}.
\newblock {\em Rev. Mod. Phys.}, 50(2):221--260, 1978.

\bibitem{neumann1927}
J.~Von~Neumann.
\newblock Wahrscheinlichkeitstheoretischer {A}ufbau der {Q}uantenmechanik.
\newblock {\em G\"ottinger Nachrichten}, 1:245--272, 1927.

\bibitem{BarnettSM1989}
S.~M. Barnett and S.~J.~D. Phoenix.
\newblock {Entropy as a measure of quantum optical correlation}.
\newblock {\em Phys. Rev. A}, 40(5):2404--2409, 1989.

\bibitem{BarnettSM1991}
S.~M. Barnett and S.~J.~D. Phoenix.
\newblock {Information theory, squeezing, and quantum correlations}.
\newblock {\em Phys. Rev. A}, 44(1):535--545, 1991.

\bibitem{vedral1997}
V.~Vedral, M.~B. Plenio, M.~A. Rippin, and P.~L. Knight.
\newblock {Quantifying Entanglement}.
\newblock {\em Phys. Rev. Lett.}, 78(12), 1997.

\bibitem{Greenberger2009}
D.~Janzing, .
\newblock {\em {Compendium of Quantum Physics} ed D.~Greenberger, K.~Hentschel, and F.~Weinert}, pages 205--209.
\newblock Springer Berlin Heidelberg, Berlin, Heidelberg, 2009.

\bibitem{Vedral2002}
V.~Vedral.
\newblock {The role of relative entropy in quantum information theory}.
\newblock {\em Rev. Mod. Phys.}, 74(1):197--234, 2002.

\bibitem{HighamBook08}
N.~J. Higham.
\newblock {\em Functions of matrices}.
\newblock Society for Industrial and Applied Mathematics (SIAM), Philadelphia,
  PA, 2008.
\newblock Theory and computation.

\bibitem{BaFaGo96}
Z.~Bai, M.~Fahey, and G.~Golub.
\newblock Some large-scale matrix computation problems.
\newblock {\em J. Comput. Appl. Math.}, 74(1-2):71--89, 1996.
\newblock TICAM Symposium (Austin, TX, 1995).

\bibitem{DongLiu94}
S.~Dong and K.~Liu.
\newblock Stochastic estimation with ${Z}_2$ noise.
\newblock {\em Phys. Lett. B}, 328(1--2):130--136, 1994.

\bibitem{Hutchinson89}
M.~F. Hutchinson.
\newblock A stochastic estimator of the trace of the influence matrix for
  {L}aplacian smoothing splines.
\newblock {\em Comm. Statist. Simul. Comput.}, 18(3):1059--1076, 1989.

\bibitem{FoPa68}
L.~Fox and I.~B. Parker.
\newblock {\em Chebyshev polynomials in numerical analysis}.
\newblock Oxford University Press, London, 1968.

\bibitem{Tr08}
L.~N. Trefethen.
\newblock Is {G}auss quadrature better than {C}lenshaw-{C}urtis?
\newblock {\em SIAM Rev.}, 50(1):67--87, 2008.

\bibitem{Cl55}
C.~W. Clenshaw.
\newblock A note on the summation of {C}hebyshev series.
\newblock {\em Math. Tables Aids Comput.}, 9:118--120, 1955.

\bibitem{weisse2006}
A.~Weisse, G.~Wellein, A.~Alvermann, and H.~Fehske.
\newblock{The kernel polynomial method}.
\newblock{\em Rev. Mod. Phys.}, 78(1):275--306, 2006.

\bibitem{Hoeffding63}
W.~Hoeffding.
\newblock Probability inequalities for sums of bounded random variables.
\newblock {\em J. Amer. Statist. Assoc.}, 58:13--30, 1963.

\bibitem{Gerschgorin31}
S.~Gerschgorin.
\newblock \"{U}ber die Abgrenzung der Eigenwerte einer Matrix.
\newblock {\em Izv. Akad. Nauk. USSR Otd. Fiz.-Mat. Nauk}, 6:749-754, 1931.

\bibitem{saad2011}
Y.~Saad. 
\newblock {\em Numerical Methods for Large Eigenvalue Problems.}
\newblock Classics in Applied Mathematics, SIAM, Philadelphia, 2011. 

\bibitem{Jo09}
C.~Johnson.
\newblock {\em Numerical solution of partial differential equations by the
  finite element method}.
\newblock Dover Publications Inc., Mineola, NY, 2009.
\newblock Reprint of the 1987 edition.

\bibitem{RevModPhys.71.S288}
A.~Zeilinger.
\newblock Experiment and the foundations of quantum physics.
\newblock {\em Rev. Mod. Phys.}, 71:S288--S297, 1999.

\bibitem{Genovese2005}
M.~Genovese.
\newblock {Research on hidden variable theories: A review of recent
  progresses}.
\newblock {\em Phys. Rep.}, 413(6):319--396, 2005.

\bibitem{RevModPhys.79.135}
P.~Kok, W.~J. Munro, K.~Nemoto, T.~C. Ralph, J.~P. Dowling, and G.~J. Milburn.
\newblock Linear optical quantum computing with photonic qubits.
\newblock {\em Rev. Mod. Phys.}, 79:135--174, 2007.

\bibitem{Gisin2007}
N.~Gisin and R.~Thew.
\newblock {Quantum communication}.
\newblock {\em Nat. Phot.}, 1:165--171, 2007.

\bibitem{keller_rubin97pra}
T.~E. Keller and M.~H. Rubin.
\newblock Theory of two-photon entanglement for spontaneous parametric
  down-conversion driven by a narrow pump pulse.
\newblock {\em Phys. Rev. A}, 56(2):1534--1541, 1997.

\bibitem{bennet_etalpra96}
C.~H. Bennett, H.~J. Bernstein, S.~Popescu, and B.~Schumacher.
\newblock Concentrating partial entanglement by local operations.
\newblock {\em Phys. Rev. A}, 53:2046--2052, 1996.

\bibitem{law_etal00prl}
C.~K. Law, I.~A. Walmsley, and J.~H. Eberly.
\newblock Continuous frequency entanglement: Effective finite hilbert space and
  entropy control.
\newblock {\em Phys. Rev. Lett.}, 84(23):5304--5307, 2000.

\bibitem{Zinchenko2010}
Y.~Zinchenko, S.~Friedland, and G.~Gour.
\newblock {Numerical estimation of the relative entropy of entanglement}.
\newblock {\em Phys. Rev. A}, 82(5):052336, 2010.

\end{thebibliography}

\end{document}